\documentclass[nospthms]{svmult}
\usepackage{amsmath}
\usepackage{amssymb}
\usepackage{amscd}
\usepackage{amsthm}

\usepackage[all]{xy}

\newtheorem{thm}{Theorem}[section]

\newtheorem{cor}[thm]{Corollary}

\theoremstyle{remark}

\newtheorem{rem}[thm]{Remark}

\theoremstyle{definition}

\numberwithin{equation}{section}

\allowdisplaybreaks[4]

\DeclareMathOperator{\Ext}{Ext}

\DeclareMathOperator{\CH}{CH}

\begin{document}

\vfuzz0.5pc
\hfuzz0.5pc 

\newcommand{\claimref}[1]{Claim \ref{#1}}
\newcommand{\thmref}[1]{Theorem \ref{#1}}
\newcommand{\propref}[1]{Proposition \ref{#1}}
\newcommand{\lemref}[1]{Lemma \ref{#1}}
\newcommand{\coref}[1]{Corollary \ref{#1}}
\newcommand{\remref}[1]{Remark \ref{#1}}
\newcommand{\conjref}[1]{Conjecture \ref{#1}}
\newcommand{\questionref}[1]{Question \ref{#1}}
\newcommand{\defnref}[1]{Definition \ref{#1}}
\newcommand{\secref}[1]{Sec. \ref{#1}}
\newcommand{\ssecref}[1]{\ref{#1}}
\newcommand{\sssecref}[1]{\ref{#1}}

\newcommand{\red}{{\mathrm{red}}}
\newcommand{\tors}{{\mathrm{tors}}}
\newcommand{\eq}{\Leftrightarrow}

\newcommand{\mapright}[1]{\smash{\mathop{\longrightarrow}\limits^{#1}}}
\newcommand{\mapleft}[1]{\smash{\mathop{\longleftarrow}\limits^{#1}}}
\newcommand{\mapdown}[1]{\Big\downarrow\rlap{$\vcenter{\hbox{$\scriptstyle#1$}}$}}
\newcommand{\smapdown}[1]{\downarrow\rlap{$\vcenter{\hbox{$\scriptstyle#1$}}$}}

\newcommand{\A}{{\mathbb A}}
\newcommand{\II}{{\mathcal I}}
\newcommand{\J}{{\mathcal J}}
\newcommand{\CO}{{\mathcal O}}
\newcommand{\CC}{{\mathcal C}}
\newcommand{\C} {{\mathbb C}}
\newcommand{\BC}{{\mathbb C}}
\newcommand{\BQ}{{\mathbb Q}}
\newcommand{\m}{{\mathcal M}}
\newcommand{\h}{{\mathcal H}}
\newcommand{\ZZ}{{\mathcal Z}}
\newcommand{\Z} {{\mathbb Z}}
\newcommand{\BZ}{{\mathbb Z}}
\newcommand{\W}{{\mathcal W}}
\newcommand{\Y}{{\mathcal Y}}
\newcommand{\T}{{\mathcal T}}
\newcommand{\BP}{{\mathbb P}}
\newcommand{\CP}{{\mathcal P}}
\newcommand{\G}{{\mathbb G}}
\newcommand{\BR}{{\mathbb R}}
\newcommand{\DD}{{\mathcal D}}
\newcommand{\LL}{{\mathcal L}}
\newcommand{\f}{{\mathcal F}}
\newcommand{\EE}{{\mathcal E}}
\newcommand{\BN}{{\mathbb N}}
\newcommand{\N}{{\mathcal N}}
\newcommand{\K}{{\mathcal K}}
\newcommand{\R} {{\mathbb R}}
\newcommand{\PP} {{\mathbb P}}
\newcommand{\BF}{{\mathbb F}}
\newcommand{\closure}[1]{\overline{#1}}
\newcommand{\EQ}{\Leftrightarrow}
\newcommand{\imply}{\Rightarrow}
\newcommand{\isom}{\cong}
\newcommand{\embed}{\hookrightarrow}
\newcommand{\tensor}{\mathop{\otimes}}
\newcommand{\wt}[1]{{\widetilde{#1}}}
\newcommand{\ol}{\overline}
\newcommand{\ul}{\underline}
\newcommand{\QQ}{{\mathcal Q}}

\newcommand{\bs}{{\backslash}}
\newcommand{\CS}{{\mathcal S}}
\newcommand{\Q}{{\mathbb Q}}
\newcommand{\F}{{\mathcal F}}

\author{Xi Chen \and James D. Lewis}

\institute{Xi Chen
\at 632 Central Academic Building, University of Alberta, Edmonton, Alberta T6G 2G1, CANADA
\email{xichen@math.ualberta.ca}
\and
James D. Lewis
\at 632 Central Academic Building, University of Alberta, Edmonton, Alberta T6G 2G1, CANADA
\email{lewisjd@ualberta.ca}}

\abstract{We prove some general density statements about the subgroup of invertible
points on intermediate jacobians; namely those points in the Abel-Jacobi image
of nullhomologous algebraic cycles on projective algebraic manifolds.
\keywords{Abel-Jacobi map, intermediate jacobian, normal function, Chow group;
1991 {\em Mathematics Subject Classification}: Primary 14C25, Secondary 14C30, 14C35.}
}

\thanks{Both authors partially supported by a grant from the 
Natural Sciences and Engineering Research Council of Canada.}

\date{July 24, 2012}

\title*{Dynamics of Special Points on Intermediate Jacobians}

\maketitle

\section{Introduction}\label{SEC003}

Let $X/\C$ be a projective algebraic manifold, $\CH^r(X)$ the Chow group of codimension $r$
algebraic cycles on $X$ (with respect to the equivalence relation of rational equivalence),
and $\CH^r_{\hom}(X)$ the subgroup of cycles that are nullhomologous  under the cycle class map
to singular cohomology with $\Z$-coefficients. Largely in relation to the celebrated
Hodge conjecture, as well as with regard to equivalence relations on algebraic cycles,
the Griffiths Abel-Jacobi map
\[
\Phi_r : \CH^r_{\hom}(X) \to J^r(X)\ {\buildrel {\rm Carlson}\over  \simeq}\ \Ext_{\rm MHS}^1\big(\Z(0),H^{2r-1}(X,\Z(r))\big),
\]
has been a focus of attention for the past 60 years. The role of the Abel-Jacobi map in
connection to the celebrated Hodge conjecture began with the work of Lefschetz in
his proof of his famous ``Lefschetz $(1,1)$ theorem''; and which inspired Griffiths
to develop his program of  updating Lefschetz's ideas as a general line of attack on
the Hodge conjecture (see \cite{Z1}, as well as \cite[Lec. 6, 12, 14]{L}).
To this day, a {\em precise} statement 
about what the image of $\Phi_r$ is in general seems rather elusive.
What we do know is that there are examples
where the image of $\Phi_r$ can be a countable set (even infinite dimensional over $\Q$)
\cite{Gr}, \cite{Cl},
or completely torsion \cite{G}. One can ask whether the image of $\Phi_r$ is always dense in $J^r(X)$,
but even that is unlikely to be true in light of some results in the
literature inspired by some of the conjectures in \cite{G-H}.
In this paper, we seek to come up with a general statement about the image of $\Phi_r$, 
which however is
modest,  is indeed is better than no statement at all. There are two key ideas exploited
in this paper, viz., the business of Lefschetz pencils and associated normal functions, and 
the  classical Kronecker's theorem  (see \cite{H-W} (Chapter XXIII)), which we
state in the following form:

{\begin{thm}\label{THM200}
Let $A = \BR^n /\BZ^n$ be a compact real torus of dimension $n$. 
For a point $p = (x_1, x_2, ..., x_n)\in A$, $\BZ p = \{ kp: k\in \BZ\}$ is dense in $A$ if and only
if $1, x_1, x_2, ..., x_n$ are linearly independent over $\BQ$.
In particular, the set
\begin{equation}\label{E034}
\big\{p\in A:  \BZ p \text{ is not dense in } A\big\}
\end{equation}
is of the first Baire category.
\end{thm}}

The main results are stated in Theorem \ref{T303} and Corollaries \ref{C33} and \ref{C666} below.

\medskip
We are grateful to our colleagues Matt Kerr, Phillip Griffiths and Chuck Doran for enlightening discussions, as well as the referee for suggesting improvements in presentation. 

\section{Some preliminaries}

All integral cohomology is intended modulo torsion.
Let $X/\C$ be a projective algebraic manifold of dimension $2m$
and $\{X_t\}_{t\in \PP^1}$ a Lefschetz pencil of hyperplane sections
of $X$ arising from a given polarization on $X$.  Let $D := \bigcap_{t\in \PP^1}X_t$ be the (smooth) base locus
and $\ol{X} = B_D(X)$, the blow-up. One has a diagram:

\begin{equation}\label{D27}
\begin{matrix} \ol{X}_U&\hookrightarrow&\ol{X}\\
&\\
\rho_U\biggl\downarrow\quad&&\quad\biggr\downarrow\rho\\
&\\
U&{\buildrel j\over \hookrightarrow}&\PP^1,
\end{matrix}
\end{equation}
where $\Sigma := \PP^1\bs U = \{t_1,...,t_M\}$ is the singular set, viz.,
where the fibers are singular Lefschetz hyperplane sections. One has a short exact sequence
of sheaves
\begin{equation}\label{999}
0 \to j_{\ast}R^{2m-1}\rho_{U,\ast}\Z \to \ol{\F}^{m,\ast}\to \ol{\J}\to 0,
\end{equation}
where 
\[
\ol{\F}^{m,\ast} = \CO_{\PP^1}\biggl(\coprod_{t\in \PP^1}
\frac{H^{2m-1}(X_t,\C)}{F^mH^{2m-1}(X_t,\C)}\biggr)\quad {\rm (canonical\ extension)},
\]
and where the cokernel sheaf $\ol{\J}$ is the sheaf of germs of normal functions. The canonical  (sometimes called the privileged) extension $\ol{\F}^{m,*}$ of the vector bundle
\[
\F^{m,*} := \CO_{U}\biggl(\coprod_{t\in U}
\frac{H^{2m-1}(X_t,\C)}{F^mH^{2m-1}(X_t,\C)}\biggr),
\]
is introduced in \cite{Z1} (as well as in the references cited there). It plays a role
in the required limiting behaviour of the group $H^0(\PP^1,\ol{\J})$ of normal functions ``at the boundary'', viz.,  at $\Sigma$. Roughly speaking then, a normal function $\nu\in 
H^0(\PP^1,\ol{\J})$ is a holomorphic cross-section,
\[
\nu : \PP^1 \to \coprod_{t\in \PP^1}J^m(X_t),
\]
where for $t\in \Sigma$, $J^m(X_t)$ are certain ``generalized'' intermediate  jacobians, and
where $\nu$ is locally liftable to a section of $\ol{\F}^{m,*}$.\footnote{There is also a horizontality condition attached to the definition of normal functions of families of projective algebraic
manifolds, which automatically holds in the Lefschetz pencil situation (see \cite[Thm. 4.57]{Z1}).} The results in \cite[Cor. 4.52]{Z1}  show that (\ref{999}) induces a short exact sequence:
\begin{equation}\label{D28}
0 \to J^m(X) \to H^0(\PP^1,\ol{\J}) \xrightarrow{\delta} 
H^1(\PP^1,j_*R^{2m-1}\rho_{U,\ast}\Z)^{(m,m)}\to 0,
\end{equation}
where it is also shown that with respect to the aforementioned polarization of $X$ defining primitive cohomology, 
\begin{equation}\label{ER}
H^1(\PP^1,j_*R^{2m-1}\rho_{U,\ast}\C) \simeq {\rm Prim}^{2m}(X,\C) \bigoplus H_v^{2m-2}(D,\C),
\end{equation}
where $H^{2m-2}_v(D,\C)  = \ker \big(H^{2m-2}(D,\C)\to H^{2m+2}(X,\C)\big)$, (induced
by the inclusion $D\hookrightarrow X$, and where
$H^1(\PP^1,j_*R^{2m-1}\rho_{U,\ast}\Z)^{(m,m)}$ are the integral
classes of Hodge type $(m,m)$ in $H^1(\PP^1,j_*R^{2m-1}\rho_{U,\ast}\Z)$,
and the fixed part $J^m(X)$ is the
Griffiths intermediate jacobian of $X$. It should be pointed out that 
there is an intrinsically defined Hodge structure on the space
$H^1(\PP^1,j_*R^{2m-1}\rho_{U,\ast}\C)$ and that (\ref{ER}) is an isomorphism
of Hodge structures  \cite{Z2}.
For $t\in U$, the Lefschetz theory  guarantees an orthogonal decomposition
\[
H^{2m-1}(X_t,\C) = H^{2m-1}(X,\C) \oplus H^{2m-1}_v(X_t,\C),
\]
where by the weak Lefschetz theorem, $H^{2m-1}(X,\C)$ is identified with its image
$H^{2m-1}(X,\C)\hookrightarrow H^{2m-1}(X_t,\C)$
and integrally speaking, $H^{2m-1}_v(X_t,\Z)$ is the space generated by the vanishing
cocycles $\{\delta_1,...,\delta_M\}$ (cf. \cite[Lec. 6, p. 71]{L}).  For fixed $t\in U$, we put
\[
J^m_v(X_t) = \Ext^1_{\rm MHS}\big(\Z(0),H^{2m-1}_v(X_t,\Z(m))\big).
\]

For each $t_i\in\Sigma$,
we recall the Picard-Lefschetz transformation
$T_i$, and formula $T_i(\gamma) = \gamma + (-1)^m(\gamma,\delta_i)\delta_i$,
where $(\delta_i,\delta_j) :=  (\delta_i,\delta_j)_{X_t} \in \Z$ is the cup product on $X_t$ 
(followed by the trace). 
Note that a lattice in $H^{2m-1}_v(X_t,\Z)$ (i.e. defining
a basis of $H^{2m-1}_v(X_t,\Q)$),
is given (up to relabelling) by a suitable subset $\{\delta_1,...,\delta_{2g}\}$, ($2g\leq M$),
of vanishing cocycles.  However we are going to choose our lattice 
generators  $\{\delta_1,...,\delta_{2g}\}$
more carefully as follows:

\medskip
$\bullet$ \ Given $\delta_1$, choose $\delta_2$ such that $(\delta_1,\delta_2) \ne 0$.
Since $(\delta_j^2) = 0\ \forall\ j=1,...,M$, it follows that $\{\delta_1, \delta_2\}$
are $\Q$-independent.

\medskip
$\bullet$ \ Next, we argue inductively on $k$ with $1\leq k\leq 2g-1$, that
\begin{equation}\label{778}
\{\delta_1,...,\delta_k,\delta_{k+1}\}\ {\rm are}\  \Q\text{\rm-independent\ and} 
\end{equation}
\[
 (\delta_{\ell},\delta_{k+1}) \ne 0\
{\rm for\ some}\ \ell \in \{1,...,k\}.
\]
Indeed if (\ref{778}) failed to hold for any or all such $k$, then in light
of the Picard-Lefschetz formula, $\{\delta_1,...,\delta_M\}$ 
would not be conjugate under the monodromy group action \cite[Lec. 6, p. 71]{L}.

\section{Main results}

The class $\delta(\nu)\in H^1(\PP^1,j_*R^{2m-1}\rho_{U,\ast}\Z)^{(m,m)}$ is called
the topological invariant or the cohomology class of the normal function $\nu$. 

\begin{thm}\label{T303} Let $\nu\in H^0(\PP^1,\ol{\J})$ be 
a normal function with nontrivial cohomology class, i.e., satisfying $\delta(\nu)\ne 0$.
Then for very general $t\in U$, the subgroup $\langle \nu(t)\rangle \subset J_v^m(X_t)$
generated by $\nu(t)$, is dense in the strong topology. In particular, the family
of rational curves in the manifold (see \cite{Z1}, Prop. 2.9):
\[
{\bf J} := \coprod_{t\in \PP^1}J_v^m(X_t),
\]
(viz., the images of non-constant holomorphic maps $\PP^1\to {\bf J}$),  is dense in the strong topology.
\end{thm}

\begin{proof} From the Picard-Lefschetz formula,
\[
N_i = \log T_i =(T_i-I),\ {\rm using}\ (T_i-I)^2=0.
\]
Now let $\nu \in H^0(\PP^1,\ol{\J})$ and $\omega\in H^0(\PP^1,\ol{\F}^m)$ 
be given. Note that 
\[
\nu : \PP^1 \to {\bf J},
\]
defines a rational curve on ${\bf J}$. Next, the images 
\[
\{[\delta_1],...,[\delta_{2g}]\} \ {\rm in}\ F^{m,\ast}H_v^{2m-1}(X_t,\C) := H_v^{2m-1}(X_t,\C)
/ F^mH_v^{2m-1}(X_t,\C),
\]
define a lattice. In terms of this lattice and modulo the fixed part $J^m(X)$,
a local lifting of $\nu$ is given by
$\sum_{j=1}^{2g}x_j(t)[\delta_j]$, for suitable real-valued functions $\{x_j(t)\}$, multivalued
on $U$. Let $T_i\nu(\omega(t))$ be the result of analytic
continuation of $\nu(\omega(t))$ counterclockwise in $\PP^1$ about $t_i$ and 
$N_i\nu(\omega(t)) = T_i\nu(\omega(t)) - \nu(\omega(t))$. 
About $t_i$, we pick up a period
\[
N_i\nu(\omega(t)) = c_i\int_{\delta_i}\omega(t),\ {\rm for\ some}\ c_i\in \Z,
\]
dependent only on $\nu$ (not on $\omega$), where we identify $\delta_i$ with its
corresponding homology vanishing cycle via Poincar\'e duality. 
Likewise in terms of the lattice description,
\[
N_i\nu(\omega(t)) = \sum_{j=1}^{2g}T_i(x_j(t))\int_{\delta_j+(-1)^m(\delta_j,\delta_i)\delta_i}
\omega(t) \ - \ \sum_{j=1}^{2g} x_j(t)\int_{\delta_j}\omega(t)
\]
\[
= \sum_{j=1}^{2g}N_i(x_j(t))\int_{\delta_j}\omega(t) \ +\  (-1)^m\biggl(
\sum_{j=1}^{2g}T_i(x_j(t))(\delta_j,\delta_i)\biggr)\cdot\int_{\delta_i}\omega(t).
\]
Thus
\begin{equation}\label{E66}
c_i = N_i(x_i(t)) +  (-1)^m\sum_{j=1}^{2g}T_i(x_j(t))(\delta_j,\delta_i),
\end{equation}
and 
\begin{equation}\label{E13}
N_i(x_j(t)) = 0\  {\rm  for \ all} \ i\ne j.
\end{equation} Hence $T_i(x_j(t)) = x_j(t)$ for all
$i\ne j$ and further, using $(\delta_i,\delta_i) = 0$, we can rewrite
equation (\ref{E66}) as:
\begin{equation}\label{E67}
c_i = N_i(x_i(t)) + (-1)^m\sum_{j=1}^{2g}x_j(t)(\delta_j,\delta_i).
\end{equation}
Note that if $N_i(x_i(t))  = 0$ for all $i$, then from the linear
system in (\ref{E67}), $x_i(t)\in \Q$ for all $i$,
and so $\delta (\nu) = 0\in H^1(\PP^1,j_{\ast}R^{2m-1}\rho_{U,\ast}\Q)$.
Now suppose that we have a nontrivial relation:
\begin{equation}\label{E14}
\sum_{j=1}^{2g}\lambda_jx_j(t) = \lambda_0,\ {\rm for\ some}\ \lambda_i\in \Q, \ \forall i, \
t\in U.
\end{equation}
Then by (\ref{E13}) and (\ref{E14}) we have
\[
\lambda_iN_i(x_i(t)) = \sum_{j=1}^{2g}\lambda_jN_i(x_j(t)) = 0.
\]
So $\lambda_i \ne 0 \Rightarrow N_i(x_i(t)) = 0$.  
Let us assume for the moment that $\lambda_1 \ne 0$. Then $N_1(x_1(t))=0$, hence from (\ref{E67}):
\begin{equation}\label{ES1}
(-1)^mc_1 = (\delta_2,\delta_1)x_2(t) + (\delta_3,\delta_1)x_3(t) +\cdots + (\delta_{2g},\delta_1)x_{2g}(t),
\end{equation}
and applying $N_2$ and (\ref{778}) we arrive at
\[
0 = N_2(c_1) =  (\delta_2,\delta_1)N_2(x_2(t)) 
 \Rightarrow N_2(x_2(t)) = 0.
\]
Hence again from (\ref{E67}):
\begin{equation}\label{ES2}
(-1)^mc_2 = (\delta_1,\delta_2)x_1(t) + (\delta_3,\delta_2)x_3(t)+\cdots+(\delta_{2g},\delta_2)
x_{2g}(t).
\end{equation}
Applying $N_3$ to both equations (\ref{ES1}) and (\ref{ES2}), and (\ref{778}) we arrive at
\[
(0,0) = \big(N_3(c_1),N_3(c_2)\big) = \big((\delta_3,\delta_1),(\delta_3,\delta_2)\big)\cdot
N_3(x_3(t)) \Rightarrow N_3(x_3(t)) = 0,
\]
and so on.  Now it may happen that $\lambda_1=0$.
Since $(\lambda_1,...,\lambda_{2g})
\ne (0,...,0)$ we can assume that $\lambda_{\ell_1} \ne 0$ for some $1\leq \ell_1 \leq 2g$.
Thus by (\ref{E14}), $N_{\ell_1}(x_{\ell_1}(t)) = 0$ and accordingly by (\ref{E67}):
\begin{equation}\label{ES3}
(-1)^mc_{\ell_1} = \sum_{j=1}^{2g}(\delta_j,\delta_{\ell_1})x_j(t).
\end{equation}
By (\ref{778}), $(\delta_{\ell_2},\delta_{\ell_1}) \ne 0$ for some $1\leq \ell_2<\ell_1$ (assuming
$\ell_1 > 1$). Applying $N_{\ell_2}$ to (\ref{ES3}), we arrive at $N_{\ell_2}(x_{\ell_2}(t)) = 0$,
and hence again by (\ref{E67}):
\[
(-1)^mc_{\ell_2} = \sum_{j=1}^{2g}(\delta_j,\delta_{\ell_2})x_j(t).
\]
Again by (\ref{778}), $(\delta_{\ell_3},\delta_{\ell_2}) \ne 0$ for some $1\leq \ell_3<\ell_2$ (assuming
$\ell_2 > 1$), and thus we can repeat this process until we get $N_1(x_1(t)) = 0$. This puts in
the situation of equation (\ref{ES1}), where the same arguments imply that $N_i(x_i(t)) = 0$ for
all $i=1,...,2g$.
\end{proof}

\begin{cor}\label{C33}  Let $V$ be a general
quintic threefold. Then the image of the Abel-Jacobi map
$AJ : \CH^2_{\hom}(V) \to J^2(V)$ is a countable dense subset of 
$J^2(V)$.
\end{cor}

\begin{proof} Let $X\subset \PP^5$ be the Fermat quintic fourfold,
and $\{X_t\}_{t\in \PP^1}$ a Lefschetz pencil of hyperplane sections
of $X$. We will assume the notation given
in diagram (\ref{D27}). For the Fermat quintic, it is easy to check that
$H^1(\PP^1,R^3\rho_{U,\ast}\Q)^{(2,2)} \ne 0$, so by the
sequence in (\ref{D28}), there exists $\nu\in H^0(\PP^1,\ol{\J})$ such
that $\delta(\nu) \ne 0 \in H^1(\PP^1,R^3\rho_{U,\ast}\Q)$ (this being
related to Griffiths' famous example \cite{Gr}). Thus by
Theorem \ref{T303} and for general $t\in \PP^1$,
the Abel-Jacobi image is dense in $J^2(X_t)$. But it is well known that the lines
in $X_t$ for general $t\in \PP^1$, deform in the universal family of
quintic threefolds in $\PP^4$. The corollary follows from this.
\end{proof}

\begin{rem}\label{RG} {\rm In light of the conjectures in \cite{G-H}, Corollary \ref{C33}
most likely does not generalize to higher degree general hypersurface
threefolds. However there is a different kind of generalization that probably
holds. Namely, let $S$ be the universal family of smooth threefolds
$\{V_t\}_{t\in S}$ of degree $d$ say in $\PP^4$. Put
\[
{\bf J}_S := \coprod_{t\in S}J^2(V_t),
\]
and
\[
{\bf J}^2_{S,{\rm inv}} := {\rm Image}\biggl(\coprod_{t\in S}\CH^2_{\hom}(V_t) 
\xrightarrow{{\rm Abel-Jacobi}}
{\bf J}_S\biggr).
\]
Then in the strong topology, we anticipate that ${\bf J}^2_{S,{\rm inv}}
\subset {\bf J}_S$ is a dense subset.}
\end{rem}

In this direction, we have the following general result.
\begin{cor}\label{C666} Let $\coprod_{{\lambda\in S_0}}W_{\lambda} \to S_0$ be a
smooth proper  family  of $2m$-dimensional projective varieties in 
some $\PP^N$ with the following property:

\medskip
\noindent
There exists a dense subset $\Sigma\subset S_0$ such that $\lambda\in \Sigma
\Rightarrow {\rm Prim}_{\rm alg}^{m,m}(W_{\lambda},\Q) \ne 0$,
where Prim is primitive cohomology with respect to the embedding $W_{\lambda}
\subset \PP^N$. Further, let us assume that $H^{2m-1}(W_{\lambda},\Q)^{\pi_1(S_0)} = 
H^{2m-1}(W_{\lambda},\Q)$ and let
\[
T := \big\{t:= (c,\lambda)\in \PP^{N,\ast}\times S_0\ \big|\ V_t := \PP^{N-1}_c\cap W_{\lambda}\
smooth,\ \&\dim V_t=2m-1\big\},
\]
with corresponding ${\bf J}^m_{T,{\rm inv}} \subset {\bf J}^m_T$ (where this jacobian
space only involves the orthogonal complement of the fixed part of a corresponding
variation of Hodge structure).  
Then in the strong topology
${\bf J}^m_{T,{\rm inv}}$ is dense in ${\bf J}^m_T$. 
\end{cor}

\begin{proof} This easily follows from the techniques of this section and
is left to the reader.
\end{proof}

\begin{rem}
{\rm  (i) The following is obvious, but certainly merits mentioning:
{\em Let us assume given the setting and assumptions
in Corollary \ref{C666}, and further assume that for all
$\lambda\in \Sigma$ and general $c$ with $t = (c,\lambda)\in T$, the $m$-th
$\Q$-Griffiths group $\big\{\CH^m_{\hom}(V_t)\big/
\CH^m_{\rm alg}(V_t)\big\}\otimes\Q = 0$. Then 
${\bf J}^m_T$ is a family of Abelian
varieties.}  (The general Hodge conjecture would imply 
in this situation that
${\bf J}^m_{T,{\rm inv}} = {\bf J}^m_T$, but we don't  yet
know this.)

\medskip
(ii) Our results say nothing about the {\em arithmetic} nature of
the invertible points on the jacobians. 
Matt Kerr pointed out to us Proposition 124  in \cite[p. 92]{K-P}, which appears to
be related to our results, and
may have some potential  in this direction; albeit it is unclear how to move
forward with this.}
\end{rem}

\end{document}